\documentclass[12pt,noindent]{article}
\usepackage{latexsym,amsmath,amsfonts}
\newtheorem{theorem}{Theorem}
\newtheorem{proposition}{Proposition}
\newtheorem{corollary}{Corollary}
\newtheorem{lemma}{Lemma}
\newtheorem{question}{Question}

\newtheorem{claim}{Claim}

\newtheorem{remark}{Remark}

\newcommand{\thmref}[1]{Theorem~\ref{thm:#1}} % Theorem tag equals ``thm''
\newcommand{\lemref}[1]{Lemma~\ref{lem:#1}} % Lemma tag equals ``lem''
\newcommand{\remref}[1]{Remark~\ref{rem:#1}} % Remark equals ``rem''
\newcommand{\claimref}[1]{Claim~\ref{claim:#1}} % Claim tag equals ``claim''
\newcommand{\secref}[1]{Section~\ref{sec:#1}} % Section = ``sec''
\newcommand{\eqnref}[1]{(\ref{eq:#1})} % Equation = ``eq''
\def\be{\begin{equation} }
\def\ee{ \end{equation}}

\def\ben{\begin{equation*}}
\def\een{\end{equation*}}
\def\bea{\begin{eqnarray}}
\def\eea{\end{eqnarray}}
\def\ee{\end{eqnarray}}
\def\bean{\begin{eqnarray*}}
\def\eean{\end{eqnarray*}}

\newcommand\ignore[1]{}

\def\N{\mathbb{N}} % naturals
\newcommand{\Prp}[2]{\mathbb{P}_{#1}\left(#2\right)} % \Prp{a}{b} yields P_a(b); a is a subscript, e.g. the initial state in a Markov Process
\newcommand{\Exp}[2]{\mathbb{E}_{#1}\left[#2\right]} % \Exp{a}{b} yields E_a[b]; a is a subscript, e.g. the initial state in a Markov Process
\newcommand{\Prpwo}[1]{\mathbb{P}_{#1}} % \Prpwo{a}{b} yields P_a; a is a subscript

\renewcommand{\Pr}[1]{\mathbb{P}\left(#1\right)} % \Pr{abc} prints P(abc) with appropriately sized parentheses
\newcommand{\bigoh}[1]{O\left(#1\right)}

\newcommand{\ohmega}[1]{\Omega\left(#1\right)}

\def\sP{\mathcal{P}}
\def\sU{\mathcal{U}}

\newcommand\QED{\ifhmode\allowbreak\else\nobreak\fi
\quad\nobreak$\Box$\medbreak}
\newcommand{\proofstart}{\par\noindent\sl Proof:\rm\enspace}
\newcommand{\proofend}{\QED\par}
\newenvironment{proof}{\proofstart}{\proofend}

\def\eps{\epsilon}
\def\Tmix{{\rm T}_{\rm mix}}

\def\Trmix{{\rm T}_{\rm rmix}}

\def\Thit{{\rm T}_{\rm hit}}

\def\bT{[0,+\infty)}
\def\dtv{d_{\rm TV}}

\begin{document}
\title{Mixing and hitting times for finite Markov chains}
\author{Roberto Imbuzeiro Oliveira\thanks{IMPA, Rio de Janeiro, RJ, Brazil, 22430-040.  Work supported by a {\em Bolsa de Produtividade em Pesquisa} and by a {\em Pronex} grant from CNPq, Brazil.    }}

\maketitle
\abstract{\noindent Let $0<\alpha<1/2$. We show that that the mixing time of a continuous-time Markov chain on a finite state space is about as large as the largest expected hitting time of a subset of the state space with stationary measure $\geq \alpha$. Suitably modified results hold in discrete time and/or without the reversibility assumption. The key technical tool in the proof is the construction of random set $A$ such that the hitting time of $A$ is a light-tailed stationary time for the chain. We note that essentially the same results were obtained independently by Peres and Sousi.}

\section{Introduction}\label{sec:setup}

The present paper is a contribution to the general quantitative theory of finite-state Markov chains that was started in
\cite{Aldous_IneqReversible} and further developed in \cite{AldousLovaszWinkler_IneqGeneral}.  The gist of those papers is that the so-called mixing time of a Markov chain is fundamentally related, in a precise quantitative sense, to hitting times and other quantities of interest.  Our main achievement is to add a new equivalent quantity to this list by showing that mixing times nearly coincide with maximum hitting times of large sets in the state space.

\begin{remark}[Important remark] The results in this paper were proven (but not made public) around May 2010. In July 2011 we learned that extremely similar results for discrete-time chains have been proven independently by Peres and Sousi \cite{PeresSousi_MyPaper}.  We then decided to submit our results, in the hope that our ideas might also be found useful and interesting. We will discuss their results at several points in our paper. Here we just mention that the main difference between the papers is the construction of the stopping time in \lemref{buildstop} (see \secref{discussion}).\end{remark}

We need to introduce some notions before we  clarify what we mean; \cite{AldousFill_Book} and
\cite{LevinPeresWilmer_Book} are our main references for the involved concepts. In this paper $E$ will always denote the finite state space of a continuous-time Markov chain with generator $Q$, with transition rates $q(x,y)$ ($x,y\in E$, $x\neq y$). Most of the time $Q$ and $E$ will be implicit in our notation. The trajectories of the chain are denoted by $\{X_t\}_{t\geq 0}$, and the law of $\{X_t\}_{t\geq 0}$ started from $x\in
E$ or from a probability distribution $\mu$ over $E$ are denoted by $\Prpwo{x}$ or $\Prpwo{\mu}$ (respectively) . For $t\geq 0$, we write:
$$p_t(x,y)\equiv \Prp{x}{X_t=y}\,\,\,(x,y\in E)$$
for the transition probability from $x$ to $y$ at time $t$. In what follows we will always assume that $Q$ is irreducible, which implies that it has a unique stationary distribution
$\pi$ and: 
$$\forall (x,y)\in E^2\,:\,\lim_{t\to+\infty}p_t(x,y) = \pi(y).$$
We can measure the {\em rate} of this convergence after we introduce
a metric over probability distributions. We choose the total
variation metric:\[\dtv(\mu,\nu) = \max_{A\subset E}|\mu(A) - \nu(A)| = \frac{1}{2}\sum_{a\in E}|\mu(a)- \nu(a)|\;\;(\mu,\nu\mbox{ prob. measures over $E$})\] and define the mixing time of $Q$ as:
$$\Tmix^Q(\delta) = \inf\{t\geq 0\,:\, \forall{x\in E},\, \dtv(p_t(x,\cdot),\pi(\cdot))\leq
\delta\}.$$

Finally, given $\emptyset\neq
A\subset E$, we may define the hitting time of $A$ as:
$$H_A\equiv \inf\{t\geq 0\,:\, X_t\in A\}.$$

\noindent {\sf Results for reversible chains.} Recall that $Q$ is {\em reversible} if $\pi(x)q(x,y) = \pi(y)q(y,x)$ for all distinct $x,y\in E$. In this setting, Aldous proved:

\begin{theorem}[Aldous, \cite{Aldous_IneqReversible}]\label{thm:aldous}There exist universal (ie. chain independent) constants $c_-,c_+>0$
such that for any irreducible, reversible, finite-state-space Markov
chain in continuous time with generator $Q$:
$$c_-\,\Thit^Q\leq \Tmix^Q(1/4)\leq c_+\,\Thit^Q$$
where $\Thit^Q\equiv \sup\{\pi(A)\,\Exp{x}{H_A}\,:\,x\in E,\, \emptyset \neq A\subset E\}.$\end{theorem}
Notice that $\Thit^Q=1$ if $Q$ consists of iid jumps at rate $1$ between states in $E$, so $\Thit^Q$ can be viewed as a measure of how ``non-iid" the chain is. Informally, the mixing time is another measure of ``non-iid-ness", and the Theorem shows that these two measures are quantitatively related in a very strong sense. We emphasize that \thmref{aldous} is part of a much larger family of universal inequalities for reversible Markov chains; see
\cite{Aldous_IneqReversible} for details.

In this paper we prove a stronger form of \thmref{aldous}. Given $\alpha>0$, let:
$$\Thit^Q(\alpha)\,\equiv \sup\{\Exp{x}{H_A}\,:\,x\in E,\, \emptyset \neq A\subset E,\, \pi(A)\geq \alpha\}.$$
Unlike $\Thit^Q$, only ``large enough" sets are
considered in this definition. We prove in \secref{main} that:
\begin{theorem}\label{thm:main}For any $0<\alpha<1/2$ there exist constants $C_+(\alpha),C_-(\alpha)>0$ depending only on $\alpha$ such that, for any irreducible continuous-time Markov chain as above:
$$C_-(\alpha)\,\Thit^Q(\alpha)\leq \Tmix^Q(1/4)\leq C_+(\alpha)\,\Thit^Q(\alpha).$$\end{theorem} 

Although similar to \thmref{aldous}, the intuitive content of \thmref{main} seems different:  instead of measures of non-iid-ness, we have a statement that says that mixing times are about as large as the expected time necessary to hit any large set, which is quite reasonable. \thmref{main} should also be easier to use in applications. The condition $\alpha<1/2$ is discussed in \secref{discussion}.

\begin{remark}\thmref{main} also holds in discrete time if $p_1(x,x)\geq 1/2$ for all $x\in E$ (use \cite[Theorem 20.3]{LevinPeresWilmer_Book}). Peres and Sousi \cite{PeresSousi_MyPaper} have shown that $p_1(x,x)\geq \beta>0$ for any fixed $\beta>0$ also suffices. Some lower bound on $p_1(x,x)$ is necessary; otherwise there are counterexamples such as large complete bipartite graphs with an edge added to one of the parts.\end{remark}

\noindent {\sf Results for non-reversible chains.} \thmref{main} and the main results of \cite{Aldous_IneqReversible} only apply to reversible chains; counterexamples can be found in that paper.  Aldous, L\'{o}vasz and Winkler \cite{AldousLovaszWinkler_IneqGeneral} developed a quantitative theory in the general case using a different notion of mixing time. Let ${M}_1([0,t])$ be the set of all probability measures over $[0,t]$ and define:
$$\Trmix^Q(\delta)\equiv \inf\left\{t\geq 0\,:\, \begin{array}{l}\exists \mu\in M_1([0,t]),\, \forall x\in E,\,\\ \dtv\left(\int_{[0,t]}\,p_s(x,\cdot)\,\mu(ds),\pi\right)\leq \delta\end{array}\right\}.$$
In discrete time, one replaces $M_1([0,t])$ with the set $M_1(\{0,\dots,t\})$ of all probability measures over $\{0,\dots,t\}$.  Aldous, L\'{o}vasz and Winkler \cite{AldousLovaszWinkler_IneqGeneral} proved an analogue to \thmref{aldous} for arbitrary Markov chains in discrete time, where $\Trmix$ replaces $\Tmix$ (their method can also be applied in continuous time). We prove an analogue of \thmref{main} in this setting:

\begin{theorem}\label{thm:general}For any $\alpha\in (0,1/2)$ there exist $C'_-(\alpha)>0,C'_+(\alpha)$ such that for any irreducible finite-state Markov chain $Q$ in continuous time:
$$C'_-(\alpha)\,\Thit^Q(\alpha) \leq \Trmix^Q(1/4)\leq C'_+(\alpha)\,\Thit^Q(\alpha).$$\end{theorem}

\begin{remark}Our proof can be easily adapted to discrete time. Peres and Sousi \cite{PeresSousi_MyPaper} have proved a variant of \thmref{general} where $\Trmix^Q(1/4)$ is replaced by another notion of time-averaged mixing, with $\mu$ a geometric distribution with success probability $1/t$.\end{remark}

\subsection{Discussion of the results}\label{sec:discussion}

Outside of potential applications to bounding mixing, Theorems \ref{thm:main} and \ref{thm:general} seem conceptually interesting. They show that mixing times are {\em natural} in that they are strongly related to hitting times, a quantity of intrinsic interest. For instance, we have the following immediate corollary of \thmref{general}.
\begin{corollary}There exists some universal $C>0$ such that for any irreducible Markov chain in discrete or continuous time,
$$\forall x\in E,\,\forall \emptyset\neq A\subset V\,:\,
\Exp{x}{H_A}\leq \frac{C\,\sup_{B\subset V,\, \pi(B)\geq 1/3}
\sup_{y\in E}\Exp{y}{H_B}}{\pi(A)}.$$ \end{corollary}

We omit the proof, which follows from $\Thit^Q\leq c\,\Trmix(1/4)\leq c'\,\Thit^Q(1/3)$ (with $c,c'>0$ universal). This result says that one may control the hitting times of small sets via those of large sets.sOther applications of (slight variants of) our theorems are considered in \cite{PeresSousi_MyPaper}.  

The limitation $\alpha<1/2$ is not clearly necessary for the Theorems to hold. However, Peres \cite{Peres_Personal} noted that one cannot allow $\alpha>1/2$. In that case one may contradict the two theorems by connecting two complete graphs $K_n$ by a single edge. In this case $\Thit(\alpha)=\bigoh{n}$ whenever $\alpha>1/2$, since any set $A$ with $\pi(A)>0$ occupies a cosntant proportion of the mass of each clique. However, mixing requires crossing the connecting edge, so $\Tmix(1/4) = \ohmega{\Trmix(1/4)}=\ohmega{n^2}$. The intersting question is then: 

\begin{question}What happens when $\alpha=1/2$?\end{question}

In 2009 Peres conjectured that $\Thit^Q(1/2)$ is also ``equivalent up to universal constant factors" to $\Tmix^Q(1/4)$ (for lazy and reversible $Q$) and $\Trmix^Q(1/4)$ (in general) \cite{Peres_Open}. We prove this result in an upcoming paper with Griffiths, Kang and Patel.
\subsection{Steps of the proof}

The main step in the proof is \lemref{buildstop}, proven in \secref{stop}. We construct there a randomized stopping time $T$, which depends on the initial distribution, such that $X_T$ has the stationary distribution. This stopping rule is the hitting time of a randomly
chosen subset $A\subset E$, where the possible values of $A$ form a chain $A_1\supset A_2\supset \dots \supset A_n$. We will see that this property property implies that we can control the tail of $H_A$ via $\Thit^Q(\alpha)$. 
We note that this stopping time was outlined in \cite[Theorem5.4]{lovaszwinkler_stoc} and \cite[Theorem
4.9]{LovaszWinkler_MixingOfDiffusions}, but it is not explicit anywhere. Moreover, results in \cite{lovaszwinkler_stoc} imply that $T$ is minimal in some sense (cf. \remref{halt}). Peres and Sousi \cite{PeresSousi_MyPaper} prove similar results via another minimal stopping rule, the so-called {\em filling rule} that was also employed in \cite{Aldous_IneqReversible,AldousLovaszWinkler_IneqGeneral}). We believe that our construction provides an interesting alternative point of view.

Ater the construction of $T$, our paper continues with the proofs of \thmref{general}, proven in \secref{general}. The elegant argument we use argument employs \lemref{buildstop} together with a simple coupling devised in the survey \cite{LovaszWinkler_MixingOfDiffusions}. The proof of \thmref{main} in \secref{main} follows a convoluted computation in \cite{Aldous_IneqReversible}, which we reproduce in order to get the sharp form we need. An Appendix presents a simple lower bound of $\Trmix^Q(\alpha/2)$ in terms of $\Thit^Q(\alpha)$.

\subsection{Acknowledgements}

We thank Yuval Peres for the counterexample in \secref{discussion} \cite{Peres_Personal} and both him and Perla Sousi for presenting \cite{PeresSousi_MyPaper} to us.
\section{A special stationary stopping time}\label{sec:stop}

We use the notation in \secref{setup}. Recall that a {\em randomized stopping time} for this chain is a $\bT$-valued random variable $T$ such that for all $t\geq 0$ the event $\{T\leq t\}$ is measurable relative to the $\sigma$-field generated by $\{X_s\}_{s\leq t}$ and an independent random variable $U$. 

\begin{lemma}\label{lem:buildstop}Suppose $\mu_0$ is a probability measure over $E$. Then there exists a randomized stopping time $T$ with
\begin{equation}\label{eq:Trandomizes}\Prp{\mu_0}{X_T=\cdot} = \pi(\cdot) \mbox { and }\Prp{\mu_0}{T>t}\leq \eps
+\frac{T^Q_{\rm hit}(\eps)}{t}\mbox{ for all }\eps\in(0,1).\end{equation} \end{lemma}

\begin{remark}The same result works (with a slightly different proof) if $\pi$ is replaced by another target distribution $\mu_1$ over $E$ and $\pi$ substitutes $\mu_1$ in the definition of $\Thit^Q(\eps)$.\end{remark}

\begin{remark}\label{rem:halt}Although we do not use this, one can show that $\Exp{\mu_0}{T}$ is minimal among all randomized stopping times with $\Prp{\mu}{X_T=\cdot}=\pi(\cdot)$. This is because our $T$ has a halting state \cite[Theorem 4.5]{LovaszWinkler_MixingOfDiffusions}.\end{remark}

\begin{remark}\label{rem:optimize}We note from the definitions that $\Thit^Q(\eps)\leq \Thit^Q/\eps$. We may plug this into \lemref{buildstop} and optimize over $\eps$ to deduce:
$$\Prp{\mu_0}{T>t}\leq \sqrt{\frac{\Thit^Q}{t}}.$$
Aldous \cite{Aldous_IneqReversible} proves a similar bound for a different stopping time, which he uses to prove \thmref{aldous}. The same proof would go through with our own $T$. Another proof of \thmref{aldous} is presented in \cite{PeresSousi_MyPaper}\end{remark}

\begin{proof}Let $n\equiv |E|$ denote the cardinality of $E$. The idea in the proof is to find a
chain of subsets $E=A_1\supset A_2\supset \dots\supset A_n\neq
\emptyset$ and numbers $p_1,\dots,p_n\geq 0$ with $\sum_ip_i=1$. We then define a random $A$ that equals $A_i$ with probability $p_i$ and define $T=H_A$.
We will then show that if $\{X_t\}_{t\geq 0}$ is a realization $\Prpwo{\mu_0}$ that is independent from $A$, then ${\rm Law}(X_{T})=\pi$. The tail
behavior of $T=H_A$ will follow automatically from the construction.\\

\noindent{\sf Notation.} For any set $\emptyset\neq S\subset E$, let $\rho_S(\cdot) =
\Prp{\mu_0}{X_{H_S}=\cdot}$ denote the harmonic measure on $S$ for the
chain started from $\mu_0$. The irreducibility of the chain implies that
$H_S<+\infty$ $\Prpwo{\mu_0}$-a.s. and therefore $\rho_S$ is a
probability measure over $E$ with support in $S$.\\

\noindent{\sf Inductive construction of $(A_i,p_i)$:} Set $A_1=E$ and choose $a_1\in A_1$ so that
$\rho_{A_1}(a_1)/\pi(a_1)$ is the maximum of $\rho_{A_1}(a)/\pi(a)$
over all $a\in A_1$. Since the $\pi$-weighted average of such ratios
satisfies:
$$\sum_{a\in A_1}\pi(a)\left(\frac{\rho_{A_1}(a)}{\pi(a)}\right) =\sum_{a\in
A_1}\rho_{A_1}(a)=1,$$ the maximal value must satisfy
$\rho_{A_1}(a_1)/\pi(a_1)\geq 1$. We then choose $p_1 =
\pi(a_1)/\rho_{A_1}(a_1)$ and note that $p_1\in [0,1]$,
$p_1\rho_{A_1}(a_1)=\pi(a_1)$ and $p_1\rho_{A_1}(a)/\pi(a)\leq 1$
for all other $a\in E\backslash\{a_1\}$.

Assume inductively that we have chosen distinct elements
$a_1,\dots,a_k\in E$ and numbers $0\leq p_1,\dots,p_k\leq 1$ such
that if $A_i=E\backslash\{a_j\,:\, 1\leq j<i\}$ ($1\leq i\leq k$),
we have the following properties:
\begin{enumerate}
\item for
all $1\leq j\leq k$, $\sum_{i=1}^kp_i\rho_{A_i}(a_j) = \pi(a_j);$
\item moreover, for $a\in E\backslash\{a_1,\dots,a_k\}$,
$\sum_{i=1}^kp_i\rho_{A_i}(a) \leq \pi(a).$\end{enumerate} Assume also that $k<n$, so that 
$A_{k+1} = E\backslash \{a_1,\dots, a_k\}$ is non-empty. We will prove that one may choose $(p_{k+1},a_{k+1})$ so as to preserve these properties for one further step. The following claim is the key:
\begin{claim}The set $\sP_{k+1}\subset [0,1]\times A_{k+1}$ of all $(p,a)$ with $\sum_{i=1}^kp_i\rho_{A_{i}}(a) + p\,\rho_{A_{k+1}}(a) = \pi(a)$
is non-empty.\end{claim}

Given the claim, we choose a pair $(p_{k+1},a_{k+1})\in \sP_{k+1}$  with {\em minimum} value of the first coordinate. Let us show that condition $2.$ above
remains valid for $a\in E\backslash \{a_1,\dots,a_{k+1}\}$. Any $a$ violating $2$ would have to satisfy:
$$\sum_{i=1}^kp_i\rho_{A_i}(a) \leq \pi(a)<p_{k+1}\rho_{A_{k+1}}(a) + \sum_{i=1}^kp_i\rho_{A_i}(a),$$
and this would imply that there is some $0\leq p<p_{k+1}$ with:
$$p\,\rho_{A_{k+1}}(a) + \sum_{i=1}^kp_i\rho_{A_i}(a)  = \pi(a)\,\, (\mbox{ie. } (p,a)\in \sP_{k+1}),$$
which would contradict the minimality of $p_{k+1}$.

To prove that condition $1.$ also remains valid, we simply observe that it certainly holds for $a_{k+1}$ and that it also holds for $a_i$, $i<k+1$, because $a_i\not\in A_{k+1}$ and therefore $\rho_{A_{k+1}}(a_i)=0$ . Hence such a choice of $p_{k+1},a_{k+1}$ preserves the induction hypothesis for one more
step.

We now prove the Claim. Notice that:
$$\frac{\sum_{a\in A_{k+1}}\pi(a)\left(\frac{\rho_{A_{k+1}}(a)}{\pi(a)}\right)}{\sum_{a\in
A_{k+1}}\pi(a)}\geq \frac{\sum_{a\in A_{k+1}}\pi(a)\left(\frac{\rho_{A_{k+1}}(a)}{\pi(a)}\right)}{\sum_{a\in
E}\pi(a)}=\sum_{a\in A_{k+1}}\rho_{A_{k+1}}(a) = 1.$$
Since the first term in the LHS is an average, there must exist some 
$a\in A_{k+1}$ with $\rho_{A_{k+1}}(a)\geq \pi(a)$, whence:
$$\sum_{i=1}^kp_i\rho_{A_{i}}(a) + \rho_{A_{k+1}}(a) \geq \pi(a).$$
Moreover, the inductive assumption $2.$ implies that $\sum_{i=1}^kp_i\rho_{A_{i}}(a) \leq\pi(a),$ so there exists some $p\in[0,1]$ with
$$\sum_{i=1}^kp_i\rho_{A_{i}}(a) + p\rho_{A_{k+1}}(a) =\pi(a),$$
which proves the claim.\\

\noindent {\sf Analysis of the construction. } Carrying the induction to its end at $k=n$ implies that there exist $p_1,\dots,p_n\in [0,1]$ and
an ordering $a_1,\dots,a_n$ of the elements of $E$ such that, if
$A_i\equiv E\backslash\{a_j:1\leq j<i\}$, then: $$\forall 1\leq
i\leq n, \pi(a_i) = \sum_{j=1}^n p_j\rho_{A_j}(a_i) = \sum_{j=1}^i
p_j\rho_{A_j}(a_i)$$ (the last identity in the RHS follows from $a_i\not\in A_j$ for $j>i$).

These are the only facts about the construction we will use in the remainder of the analysis. We now prove some consequences of these facts. First notice that:
$$\sum_{j=1}^n p_j = \sum_{j=1}^np_j\,\sum_{i=1}^n\rho_{A_j}(a_i)= \sum_{i=1}^n\sum_{j=1}^np_j\rho_{A_j}(a_i) = \sum_{i=1}^n\pi(a_i) = 1,$$
which implies that the $p_i$ form a probability distribution over
$\{1,\dots,n\}$. Moreover, the same line of reasoning implies that
for all $k\in\{1,\dots,n\}$:
\begin{equation}\label{eq:sizeofsetp_i}\sum_{j=1}^kp_j \geq
\sum_{j=1}^kp_j\,\left(\sum_{i=1}^k\rho_{A_j}(a_i)\right)=\sum_{i=1}^k\sum_{j=1}^kp_j\rho_{A_j}(a_i)
= \sum_{i=1}^k\pi(a_i) = 1-\pi(A_{k+1}),\end{equation} where
$A_{n+1}=\emptyset$ by definition.

We now define our randomized stopping time as $T=H_A$, where the
choice of $A$ is independent of the realization of the chain and
$\Pr{A=A_i}=p_i$, $1\leq i\leq n$. Notice that $A\neq \emptyset$,
hence $T<+\infty$ almost surely. Moreover, it is easy to check that $\Prp{\mu_0}{X_T = \cdot}=\pi(\cdot)$, as desired.

To finish, we bound the upper tail of $T$. Given $\eps\in (0,1)$, let $j(\eps)$ be
the largest $j\in [n+1]$ with $\pi(A_j)\geq \eps$ (recall our convention $A_{n+1}=\emptyset$). Since the $A_i$'s form a decreasing chain,
\eqnref{sizeofsetp_i} implies:$$\Prp{\mu_0}{\pi(A)\geq \eps} =
\sum_{i=1}^{j(\eps)}\Pr{A=A_i} = \sum_{i=1}^{j(\eps)}p_i\geq 1 - \pi(A_{j(\eps)+1})\geq 1-\eps.$$ Moreover, $j\leq j(\eps)$ imples $A_j\supset A_{j(\eps)}$. We deduce:
\begin{eqnarray*}\Prp{\mu_0}{T>t}&\leq &\Prp{\mu_0}{\pi(A)<\eps} +
\Prp{\mu_0}{H_A>t\mid \pi(A)\geq \eps}\\ & \leq &  \eps + 
\Prp{\mu_0}{H_{A_{j(\eps)}}>t}\\
&\leq & \eps + \frac{\Exp{\mu_0}{H_{A_{j(\eps)}}}}{t} \leq  \eps+ \frac{\Thit^Q(\eps)}{t}.\end{eqnarray*}\end{proof}

\section{Mixing of non-reversible chains}\label{sec:general}
In this section we prove \thmref{general}. 

\begin{proof}[of \thmref{general}] The lower bound on $\Trmix^Q(\alpha)$ follows easily from the ideas in \cite{AldousLovaszWinkler_IneqGeneral}. We give a proof in the Appendix for completeness. For the upper bound, we proceed as follows. Define:
$$\overline{d}_r(t) = \inf\limits_{\mu\in M_1([0,t])}\sup_{x,z\in E}\,\dtv\left(\int_0^t\,p_s(x,\cdot)\,\mu(ds),\int_0^t\,p_s(z,\cdot)\,\mu(ds)\right).$$
\begin{claim}\label{claim:submult}For all $t\geq 0$, $$\overline{d}_r(kt)\leq \overline{d}_r(t)^k.$$\end{claim}
\begin{proof}[of the Claim]~A standard compactness argument shows that there exists a measure $\mu$ which achieves the infimum in the definition of $\overline{d}_r(t)$. Let $M$ be the discrete time Markov chain whose transition probabilities are given by:
\begin{equation}\label{eq:defm}m(x,y)\equiv \int_{0}^t\,p_s(x,y)\,\mu(ds),\;(x,y)\in E^2.\end{equation}
Define:
$$\overline{d}_M(k)\equiv \sup_{(x,y)\in E^2}d_{\rm TV}(m_t(x,\cdot),m_t(y,\cdot))$$
where $m_t$ is the transition probability for $t$ steps of $m$. Notice that $\overline{d}_M(1) = \overline{d}_r(t)$ by the choice of $\mu$. Moreover, $\overline{d}_r(kt)\leq \overline{d}_M(k)$ because $k$ steps of $M$ correspond to replacing $\mu$ in \eqnref{defm} by its $k$-fold convolution with itself $\mu^{* t}$. Lemma 4.12 in \cite{LevinPeresWilmer_Book} implies that 
$$\overline{d}_r(kt)\leq \overline{d}_M(k)\leq \overline{d}_M(1)^k= \overline{d}_r(t)^k.$$\end{proof}

Notice that $\overline{d}_r(t)\leq 1/4$ implies $\Trmix^Q(1/4)\leq t$. We will spend most of the rest of the proof proving that for all irreducible Markov chains $Q$,
\begin{equation}\label{eq:goalgeneral}\mbox{\bf Goal: }\overline{d}_r\left(c(\alpha)\,\Thit^Q(\alpha)\right)\leq 1 - \delta(\alpha),\end{equation}
where $c(\alpha),\delta(\alpha)>0$ depend only on $\alpha\in (0,1/2)$. Applying the Claim with $t=c(\alpha)\,\Thit^Q(\alpha)$ and $k=k(\alpha)$ such that $(1-\delta(\alpha))^k\leq 1/4$ we may then deduce that
$$\Trmix^Q(\alpha)\leq C_+(\alpha)\, \Thit^Q(\alpha)\mbox{ where }C_+(\alpha)=k(\alpha)\,c(\alpha)\mbox{ depends only on $\alpha$,}$$ 
which is the desired result.

Given $x,z\in E$, we let $\{X_t\}_{t\geq 0}$ and $\{Z_t\}_{t\geq 0}$ denote trajectories of $Q$ started from $x$ and $z$ (respectively). Let $T_x$, $T_z$ be obtained from \lemref{buildstop} for $\mu_0=\delta_x$ and $\delta_z$ (resp.). Clearly, 
$${\rm Law}(X_{T_x}) = {\rm Law}(Z_{T_z}) = \pi.$$
Sample $\sU$ uniformly from $[0,t]$ and independently from the two chains. The Markov property and the stationarity of $\pi$ imply:
$${\rm Law}(X_{T_x+\sU}) = {\rm Law}(Z_{T_z+\sU}) = \pi.$$ 
Now fix some $t\geq 0$ and define $$\sU_x\equiv (T_x+\sU)\mod t\mbox{ and }\sU_z = (T_z+\sU)\mod t.$$ Notice that $\sU_x$ is uniform over $[0,t]$, independently from $\{X_t\}_{t\geq 0}$, and similarly for $\sU_z$. Hence:
$${\rm Law}(X_{\sU_x}) = \int_0^t\,p_s(x,\cdot)\,\mu(ds)\mbox{ and } {\rm Law}(Z_{\sU_z}) = \int_0^t\,p_s(z,\cdot)\,\mu(ds),$$
where $\mu$ is uniform over $[0,t]$. Therefore,
\begin{eqnarray}\nonumber\dtv\left(\int_0^t\,p_s(x,\cdot)\,\mu(ds),\int_0^t\,p_s(z,\cdot)\,\mu(ds)\right) &=& \dtv({\rm Law}(X_{\sU_x}),{\rm Law}(Z_{\sU_z}))\\ \label{eq:RHStailgeneral}&\leq &\dtv({\rm Law}(X_{\sU_x}),{\rm Law}(X_{T_x+\sU})) \\ \nonumber & & + \dtv({\rm Law}(Z_{\sU_z}),{\rm Law}(Z_{T_z+\sU}))\end{eqnarray}
by the triangle inequality and the previous remarks. 
We now show that: \begin{equation}\label{eq:TailUx}\dtv({\rm Law}(X_{\sU_x}),{\rm Law}(X_{T_x+\sU})) \leq \alpha + 2\sqrt{\frac{\Thit^Q(\alpha)}{t}}.\end{equation} This is of course trivial if $t<\Thit^Q(\alpha)$, so we assume the opposite is true. The coupling characterization of total variation distance implies that for any $\lambda\in (0,1)$:
\begin{eqnarray*}\dtv({\rm Law}(X_{\sU_x}),{\rm Law}(X_{T_x+\sU}))&\leq &\Prp{x}{X_{\sU_x}\neq X_{T_x+\sU}} \\ &\leq &\Prp{x}{\sU>t-T_x}\\ &\leq &\Prp{x}{T_x>\lambda\,t} + \Pr{(1-\lambda)t\leq \sU\leq t}\\ \mbox{(use \lemref{buildstop})}&=& \alpha + \frac{\Thit^Q(\alpha)}{\lambda t} +\lambda\end{eqnarray*}
Choosing $\lambda = \sqrt{\Thit^Q(\alpha)/t}$  gives \eqnref{TailUx}. We plug this and the corresponding statement for $Z_{T_x+\sU}$ into \eqnref{RHStailgeneral} to deduce:
$$\dtv\left(\int_0^t\,p_s(x,\cdot)\,\mu(ds),\int_0^t\,p_s(z,\cdot)\,\mu(ds)\right) \leq 2\alpha + 4\sqrt{\frac{\Thit^Q(\alpha)}{t}}.$$
Now recall that $\alpha<1/2$ and take 
$$t = t(\alpha)\equiv \frac{64\,\Thit^Q(\alpha)}{(1-2\alpha)^2}.$$
For this value of $t$, we have:
$$\dtv\left(\int_0^t\,p_s(x,\cdot)\,\mu(ds),\int_0^t\,p_s(z,\cdot)\,\mu(ds)\right) \leq \frac{1+2\alpha}{2}.$$
Since $x,z$ are arbitrary, we deduce \eqnref{goalgeneral} with $c(\alpha) = 64/(1-2\alpha)^2$ and $\delta(\alpha) = (1-2\alpha)/2$.\end{proof}

\section{Mixing of reversible chains}\label{sec:main}

We now prove \thmref{main}.

\begin{proof}[of \thmref{main}] Notice that $\Tmix^Q(\alpha)\geq \Trmix^Q(\alpha)$, so the lower bound in the Appendix also applies here. For the upper bound, we first define:
$$\overline{d}(t)\equiv \sup_{x,z\in E}\,\dtv(p_t(x,\cdot),p_t(z,\cdot)).$$
It is well-known that $\overline{d}$ is submultiplicative \cite[Chapter 2]{AldousFill_Book} and that $\overline{d}(t)\leq 1/4$ implies $\Tmix^Q(1/4)\leq t$. In light of this, we need to show that:
\begin{equation}\label{eq:goalmain}\mbox{\bf Goal: }\overline{d}\left(c(\alpha)\,\Thit^Q(\alpha)\right)\leq 1 - \delta(\alpha),\end{equation}
where $c(\alpha),\delta(\alpha)>0$ depend only on $\alpha\in (0,1/2)$.\\

\noindent {\sf Basic definitions for the proof.} Let $U>L>0$ (we will choose their values later).  Fix a pair $x,z\in E$ and let $\{X_t\}_{t\geq 0}$ and $\{Z_t\}_{t\geq 0}$ denote trajectories of $Q$ started from $x$ and $z$ (respectively). Also let $T_x,T_z$ be the randomized stopping times given by \lemref{buildstop} for the $X$ and $Z$ processes, and define $\eta_x,\eta_z$ to be the probability distributions of $(X_{T_x},T_x)$ and $(Z_{T_z},T_z)$ over $E\times [0,+\infty)$. Finally, we let $f_x(a)\equiv \Prp{x}{X_{T_x}=a,T_x\leq L}$ and $f_z(a)=\Prp{z}{Z_{T_z}=a,T_z\leq L}$ ($a\in E$).\\

\noindent {\sf Estimating total variation distance.} Recall:
$$\dtv(p_t(x,\cdot),p_t(z,\cdot)) = \frac{1}{2}\sum_{a\in E}|p_t(x,a)-p_t(z,a)|$$
Notice that:
$$p_t(x,a) =\Prp{x}{X_t=a, T_x\leq L} + \Prp{x}{X_t=a,T_x>L},$$
and similarly for $p_t(z,a)$. Therefore,
\begin{eqnarray}\nonumber \dtv(p_t(x,\cdot),p_t(z,\cdot)) &\leq &\frac{1}{2}\sum_{a\in E}|\Prp{x}{X_t=a, T_x\leq L}  - \Prp{z}{Z_t=a, T_z\leq L}|\\ \nonumber & & + \frac{1}{2}\sum_{a\in E}|\Prp{x}{X_t=a, T_x>L}  - \Prp{z}{Z_t=a, T_z>L}|\\ \nonumber &\leq &\frac{1}{2}\sqrt{\sum_{a\in E}\frac{\left(\Prp{x}{X_t=a, T_x\leq L}  - \Prp{z}{Z_t=a, T_z\leq L}\right)^2}{\pi(a)}}\\ \label{eq:L2CS1} & & + \frac{1}{2}\sum_{a\in E}|\Prp{x}{X_t=a, T_x>L}  - \Prp{z}{Z_t=a, T_z>L}|.\end{eqnarray}
 where the last line uses the Cauchy Schwartz inequality. We may further bound:
\begin{eqnarray*}\sum_{a\in E}|\Prp{x}{X_t=a, T_x>L}  - \Prp{z}{Z_t=a, T_z>L}|&\leq &\sum_{a\in E}\Prp{x}{X_t=a,T_x>L}\\& & +\sum_{a\in E}\Prp{z}{Z_t=a,T_z>L}\\ &\leq& \Prp{x}{T_x>L} + \Prp{z}{T_z>L},\end{eqnarray*}
and plugging this into \eqnref{L2CS1} gives the inequality:
\begin{eqnarray}\nonumber \dtv(p_t(x,\cdot),p_t(z,\cdot)) &\leq &\frac{1}{2}\sqrt{\sum_{a\in E}\frac{\left(\Prp{x}{X_t=a, T_x\leq L}  - \Prp{z}{Z_t=a, T_z\leq L}\right)^2}{\pi(a)}} \\ \label{eq:L2CS}& & + \frac{\Prp{x}{T_x>L} + \Prp{z}{T_z>L}}{2}.\end{eqnarray}

\noindent {\sf Averaging.} Our next step is to average the LHS and RHS of \eqnref{L2CS} over $t\in [L,U]$. Since $\dtv(p_t(x,\cdot),p_t(z,\cdot))$ is decreasing in $t$ \cite{LevinPeresWilmer_Book}, the distance at time $t=U$ is at most this average. We use concavity to move the averaging inside the square root and deduce:
\begin{multline}\label{eq:L2CSavg}\dtv(p_U(x,\cdot),p_U(z,\cdot))\leq \frac{1}{U-L}\int_{L}^U\dtv(p_t(x,\cdot),p_t(z,\cdot))\,dt\\ \leq \frac{1}{2}\sqrt{\frac{1}{U-L}\int_{L}^U\sum_{a\in E}\frac{\left(\Prp{x}{X_t=a, T_x\leq L}  - \Prp{z}{Z_t=a, T_z\leq L}\right)^2}{\pi(a)}}\\ + \frac{\Prp{x}{T_x>L} + \Prp{z}{T_z>L}}{2}.\end{multline}

\noindent {\sf The term inside the square root.} Define $E_L\equiv E\times [0,L]$. By the strong Markov property:
\begin{eqnarray*}\Prp{x}{X_t=a,T_x\leq L}^2 &=& \left(\int_{E_L}\,p_{t-s}(u,a)\,d\eta_x(u,s)\right)^2 \\ &=& \int_{E_L}\int_{E_L}\,p_{t-s}(u,a)p_{t-s'}(u',a)\,d\eta_x(u,s)d\eta_x(u',s').\end{eqnarray*}
By reversibility, we may rewrite the integrand in the RHS as
$$p_{t-s}(u,a)\pi(a)p_{t-s'}(a,u')/\pi(u'),$$ which implies that:
\begin{multline*}\sum_{a\in E}\frac{\Prp{x}{X_t=a,T_x\leq
L}^2}{\pi(a)} \\ = \int_{E_L}\int_{E_L}\,\left(\sum_{a\in
E}\frac{p_{t-s}(u,a)p_{t-s'}(a,u')}{\pi(u')}\right)\,d\eta_x(u,s)d\eta_x(u',s')\\= \int_{E_L}\int_{E_L}
\frac{p_{2t-s'-s'}(u,u')}{\pi(u')}\,d\eta_x(u,s)d\eta_x(u',s').\end{multline*}
Integrating over $t$ (with the change of variables $t' = 2t-s-s'$),
we find that:
\begin{multline}\label{eq:trocavar}\frac{1}{U-L}\int_{L}^{U}\sum_{a\in
E}\frac{\Prp{x}{X_t=a,T_x\leq L}^2}{\pi(a)}\,dt \\ =
\int_{E_L}\int_{E_L}\left(\frac{1}{2U-2L}\int_{2L-s-s'}^{2U-s-s'}\frac{p_{t'}(u,u')}{\pi(u')}\,dt'\right)\,d\eta_x(u,s)d\eta_x(u',s')\\
\leq \int_{E_L}\int_{E_L}\left(\frac{1}{2U-2L}\int_{0}^{2U}\frac{p_{t'}(u,u')}{\pi(u')}\,dt'\right)\,d\eta_x(u,s)d\eta_x(u',s')\end{multline}
where the last inequality follows from the fact that
$[2L-s-s',2U-s-s']\subset [0,2U]$, which holds for all $s,s'$ in the range considered. With this the bracketed term becomes independent of $s$, which may be integrated out. Since:
$$\int_{\{u\}\times [0,L]}\,d\eta_x(u,s)=f_x(u)\leq \pi(u),$$
 we obtain:
\begin{multline}\frac{1}{U-L}\int_{L}^{U}\sum_{a\in
E}\frac{\Prp{x}{X_t=a,T_x\leq L}^2}{\pi(a)}\,dt \\ \leq  \sum_{u,u'\in
E}\frac{f_x(u)f_x(u')}{\pi(u')}\left(\frac{1}{2U-2L}\int_{0}^{2U}p_{t'}(u,u')\,dt'\right)\\
\leq \sum_{u,u'\in
E}\frac{f_x(u)f_x(u')}{\pi(u')}\left(\frac{1}{2U-2L}\int_{2L}^{2U}p_{w}(u,u')\,dw\right)
 \\ + \sum_{u,u'\in
E}\frac{\pi(u)}{2U-2L}\int_{0}^{2L}p_{w}(u,u')\,dw\\ \leq \sum_{u,u'\in
E}\frac{f_x(u)f_x(u')}{\pi(u')}\left(\frac{1}{2U-2L}\int_{2L}^{2U}p_{w}(u,u')\,dw\right)
+ \frac{L}{U-L},\end{multline}
as well as a similar bound for $z$. On the other hand,  starting from the formula:
\begin{multline*}\Prp{x}{X_t=a,T_x\leq L}\,\Prp{z}{Z_t=a,T_z\leq L}  \\ = \int_{E_L}\int_{E_L}\,p_{t-s}(u,z)p_{t-s'}(u',z)\,d\eta_x(u,s)d\eta_z(u',s')\end{multline*}averaging over $t\in [L,U]$ and using $[2L-s-s',2L+2U-s-s']\supset [2L,2U]$, we may obtain:
\begin{eqnarray*}\frac{1}{U-L}\int_{L}^{U}\sum_{a\in
E}\frac{\Prp{x}{X_t=a,T_x\leq L}\Prp{z}{Z_t=a,T_z\leq
L}}{\pi(z)}\,dt\\ \geq  \sum_{u,u'\in
E}\frac{f_x(u)f_z(u')}{\pi(u')}\left(\frac{1}{2U-2L}\int_{2L}^{2U}p_{w}(u,u')\,dw\right).\end{eqnarray*} Combining these bounds we obtain\begin{eqnarray*}\sum_{z\in
E}\frac{1}{U-L}\int_{L}^{U}\frac{(\Prp{x}{X_t=z,T_x\leq L}
-\Prp{x}{Z_t=z,T_z\leq L})^2}{\pi(z)}\\ \leq \sum_{u,u'\in
E}(f_x(u)-f_z(u))\left(\frac{f_x(u')-f_z(u')}{\pi(u')}\right)\left(\frac{1}{2U}\int_{2L}^{2U}p_{w}(u,u')\,dw\right)
+ \frac{2L}{U-L}.\end{eqnarray*}

To bound the sum in the RHS, we notice again that
$f_x(\cdot),f_z(\cdot)\leq \pi(\cdot)$, and also that for all
$u\in E$, $\sum_{u'}p_{w}(u,u')=1$. Hence
\begin{multline*}\sum_{u,u'\in
E}(f_x(u)-f_z(u))\left(\frac{f_x(u')-f_z(u')}{\pi(u')}\right)\left(\frac{1}{2U-2L}\int_{2L}^{2U}p_{w}(u,u')\,dw\right)\\ \leq
\sum_{u\in E}|f_x(u)-f_z(u)|.\end{multline*} Now recall that
$$f_x(u) = \Prp{x}{X_{T_x}=u,T_x\leq L} = \pi(u) - \Prp{x}{X_{T_x}=u,T_x>L}$$ and
similarly for $z$, so that $$\sum_{u\in
E}|f_x(u)-f_z(u)| =\sum_{a\in E}|\Prp{x}{X_{T_x}=a,T_x>L}
-\Prp{z}{Z_{T_z}=a,T_z>L}|.$$

We deduce that the term inside the square root in \eqnref{L2CSavg} is bounded by:
\begin{multline*}\frac{1}{U-L}\int_{L}^{U}\sum_{a\in
E}\frac{(\Prp{x}{X_t=a,T_x\leq L}-\Prp{z}{Z_t=a,T_z\leq L})^2}{\pi(z)}\,dt\\ \leq \sum_{a\in E}|\Prp{x}{X_{T_x}=a,T_x>L}
-\Prp{z}{Z_{T_z}=a,T_z>L}| + \frac{2L}{U-L}\\ \leq \sum_{a\in E}\Prp{x}{X_{T_x}=a,T_x>L}
 + \sum_{a\in E}\Prp{z}{Z_{T_z}=a,T_z>L} + \frac{2L}{U-L}\\ \leq \Prp{x}{T_x>L} + \Prp{z}{T_z>L} + \frac{2L}{U-L}.\end{multline*}

\noindent{\sf Wrapping up.} We now plug this previous inequality into \eqnref{L2CSavg} to deduce:
\begin{multline*}\dtv(p_U(x,\cdot),p_U(z,\cdot))\\ \leq \frac{1}{2}\sqrt{\Prp{x}{T_x>L} + \Prp{z}{T_z>L} + \frac{2L}{U-L}} \\ + \frac{\Prp{x}{T_x>L} + \Prp{z}{T_z>L}}{2}. \end{multline*}
If the quantity inside the square root is $<1$, we get another upper bound:
\begin{equation}\label{eq:sqrt<1}\dtv(p_U(x,\cdot),p_U(z,\cdot))\leq \sqrt{\Prp{x}{T_x>L} + \Prp{z}{T_z>L} + \frac{2L}{U-L}}\end{equation}
Now by \lemref{buildstop}
$$\Prp{x}{T_x>L}+\Prp{z}{T_z>L}\leq 2\alpha + 2\frac{\Thit^Q(\alpha)}{L}$$
so choosing 
$$L= \frac{8\Thit^Q(\alpha)}{1-2\alpha}\mbox{ and }U = \left[\frac{8}{1-2\alpha} +\left(\frac{8}{1-2\alpha}\right)^2 \right]\,\Thit^Q(\alpha)$$
we obtain:
$$ \Prp{x}{T_x>L}+\Prp{z}{T_z>L}+ \frac{2L}{U-L}\leq \frac{1+2\alpha}{2}<1.$$
Thus the condition for \eqnref{sqrt<1} is satisfied, and we have the bound:
$$\dtv(p_U(x,\cdot),p_U(z,\cdot))\leq \sqrt{\frac{1+2\alpha}{2}}.$$
Since $x,z\in E$ are arbitrary, we deduce:
$$\overline{d}\left(\left[\frac{8}{1-2\alpha} +\left(\frac{8}{1-2\alpha}\right)^2 \right]\,\Thit^Q(\alpha)\right)\leq 1 - \left(1-\sqrt{\frac{1+2\alpha}{2}}\right),$$
which has the form requested in \eqnref{goalmain}.\end{proof}

\appendix

\section*{Appendix: the lower bound}

In this section we prove the lower bound part of the main theorems. As above, $Q$ is a irreducible continuous-time Markov chain with state space $E$ and stationary distribution $\pi$. The trajectories of the chain are denoted by $\{X_t\}_{t\geq 0}$

\begin{proposition}\label{prop:appendix}For any $\alpha\in (0,1)$, $\Thit^Q(\alpha)\leq c(\alpha)\Trmix^Q$ where $c(\alpha)>0$ depends only on $\alpha$.\end{proposition}
\begin{proof}It follows from \claimref{submult} that:
$$\Trmix^Q(1/2^k)\leq k\,\Trmix^Q(1/4).$$
In particular, 
$$\Trmix^Q(\alpha)\leq (\log_2(1/\alpha) + 1)\,\Trmix^Q(1/4).$$
Thus it suffices to show that $\Thit^Q(\alpha)\leq (2/\alpha)\,\Trmix^Q(\alpha/2)$.

Fix $A\subset V$ with measure $\pi(A)\geq \alpha$ and $x\in V$. By the definition of $\Trmix^Q(\alpha/2)$ and a simple compactness argument, there exists a distribution supported on $[0,\Trmix^Q(\alpha/2)]$ such that if $\sU$ has this distribution and is independent from $\{X_t\}_t$, 
$$\dtv({\rm Law}(X_\sU),\pi) \leq 1 - \alpha/2.$$
As a result,
$$\Prp{x}{X_\sU\not\in A} \leq 1 - \pi(A) + \dtv({\rm Law}(X_\sU),\pi)\leq 1 - \frac{\alpha}{2}.$$
Since $\sU$ is supported in $[0,\Trmix^Q(\alpha/2)]$, $$\{H_A\geq \Trmix^Q(\alpha/2)\}\subset\{X_\sU\not\in A\},$$ and we deduce:
\begin{equation}\label{eq:induct}\forall x\in V,\, \forall A\subset V\mbox{ with }\pi(A)\geq \alpha\,:\,\Prp{x}{H_A\geq \Trmix^Q(\alpha/2)} \leq 1 - \frac{\alpha}{2}.\end{equation}
Let us use this to show that $\Exp{x}{H_A}\leq (2/\alpha)\,\Trmix^Q(\alpha/2)$ for all $x$ and $A$ as above. Let $k\in \N\backslash\{0\}$ and denote by $\Lambda_k$ the law of $X_{(k-1)\Trmix^Q(\alpha/2)}$ conditioned on $\{H_A\geq (k-1)\Trmix^Q(\alpha/2)\}$. By \eqnref{induct},
$$\Prp{\Lambda_k}{H_A\geq \Trmix^Q(\alpha/2)}\leq 1- \frac{\alpha}{2},$$
whereas by the Markov property,
\begin{eqnarray*}\Prp{x}{H_A\geq k\Trmix^Q(\alpha/2)}&\leq & \Prp{x}{H_A\geq (k-1)\Trmix^Q(\alpha/2)}\Prp{\Lambda_k}{H_A\geq\Trmix^Q(\alpha/2)}\\ &\leq & \left(1 -\frac{\alpha}{2}\right)\,\Prp{x}{H_A\geq (k-1)\Trmix^Q(\alpha/2)}\\ \mbox{(...induction...)}&\leq & \left(1 -\frac{\alpha}{2}\right)^k\end{eqnarray*}
We deduce:
$$\frac{\Exp{x}{H_A}}{\Trmix^Q(\alpha/2)}\leq \sum_{k\geq 0}\left(1 -\frac{\alpha}{2}\right)^k = \frac{2}{\alpha}$$
Since $x\in V$ and $A\subset V$ with $\pi(A)\geq \alpha$ were arbitrary, this finishes the proof.\end{proof}

\end{document}